\documentclass[reqno, a4paper, 10.5pt]{amsart}
\usepackage{amssymb,url, color, mathrsfs, multirow, makecell, array, caption}
\usepackage[colorlinks=true, bookmarks=true, pdfstartview=FitH, pagebackref=true, linktocpage=true]{hyperref}
\usepackage[short,nodayofweek]{datetime}
\usepackage[pagewise,running,mathlines,displaymath, switch]{lineno}
\usepackage{times}
\usepackage{indentfirst, tabularx}
\usepackage[table]{xcolor}
\usepackage{float, hhline, colortbl}
\usepackage{graphicx}
\usepackage{longtable}
\usepackage{empheq}
\usepackage{aligned-overset}

\usepackage[framemethod=TikZ]{mdframed}

\hypersetup{
	colorlinks = true,
	linkcolor = magenta,
	citecolor = blue,
}

\surroundwithmdframed[
topline=false,
rightline=true,
bottomline=false,
leftmargin=\parindent,
skipabove=\medskipamount,
skipbelow=\medskipamount
]{theorem}

\theoremstyle{plain}
\numberwithin{equation}{section}
\newtheorem{theorem}{Theorem}[section]
\newtheorem{lemma}[theorem]{Lemma}

\newtheorem{proposition}[theorem]{Proposition}

\theoremstyle{remark}

\def\ve{\varepsilon}
\def\R{\mathbf R}
\def\N{\mathbb N}
\def\ov{\overline}
\def\0{\textbf 0}

\renewcommand{\S}{\mathbf S}

\numberwithin{equation}{section}

\makeatletter
\def\@cite#1#2{[\textbf{#1}\if@tempswa, #2\fi]}
\makeatother

\title[Moving sphere approach to a general weighted integral equation]{Moving sphere approach to a general weighted integral equation}

\def\cfac#1{\ifmmode\setbox7\hbox{$\accent"5E#1$}\else\setbox7\hbox{\accent"5E#1}\penalty 10000\relax\fi\raise 1\ht7\hbox{\lower1.05ex\hbox to 1\wd7{\hss\accent"13\hss}}\penalty 10000\hskip-1\wd7\penalty 10000\box7 }

\author[Q. N.T. L\^e]{Qu\`ynh N.T. L\^e}
\address[Qu\`ynh. N.T. L\^e]{
	University of Economics \& Business, Vietnam National University, Hanoi, Vietnam}
\email{\href{mailto: Q. N. T. L\^e <ngocquynhlt@vnu.edu.vn>}{ngocquynhlt@vnu.edu.vn}
}

\author[T.-T. Nguyen]{Tien-Tai Nguyen}
\address[Tien-Tai Nguyen]{
	University of Science, Vietnam National University, Hanoi, Vietnam\\
	ORCID iD: 0000-0002-5865-8134}
\email{\href{mailto: T.-T. Nguyen <nttai.hus@vnu.edu.vn>}{nttai.hus@vnu.edu.vn}
}

\allowdisplaybreaks

\begin{document}

\begin{abstract}
Let $p$ be positive and $n \geq 3$ be an integer.  Let $f(\cdot,\cdot): \R_+\times \R_+\to \R_+$ be a  continuous function. In this paper, we are concerned with positive solutions to the following integral equation
\[
u(x)= \int_{\R^n} |x-y|^p f(|y|,u(y)) dy \quad\text{in }\R^n\setminus\{\0\}.
\]
By imposing some suitable conditions on $f$, we obtain the radially symmetry property of positive solutions to the above equation by using the method of moving spheres in integral form.
\end{abstract}
	
\date{\bf \today \, at \currenttime}
	
	
	\maketitle
 \tableofcontents
\section{Introduction}

Let $n\geq 3$ be an integer. Let $k$ be a smooth function in $\R^n$, the equation
\begin{equation}\label{Eq2ndOrder}
-\Delta u=k(x)u^{\frac{n+2}{n-2}} \quad\text{in }\R^n,
\end{equation}
is closely related to the famous Yamabe problem as well as the prescribing scalar curvature problem on the unit sphere $\S^n$. We refer to \cite{GNN79, GS81, BVV91} for celebrated results on the existence of (non)radial solutions and their properties to Eq. \eqref{Eq2ndOrder}. In \cite{JLX08}, Jin, Li and Xu study the genreal second-order elliptic equation with continuous function $f(\cdot,\cdot): \R_+\times \R_+\to \R_+$ 
\begin{equation}\label{EqDelta}
-\Delta u=f(|x|,u) \quad\text{and}\quad u>0 \text{ in }\R^n\setminus\{\0\}.
\end{equation}
Imposing the condition:
\begin{equation}\label{Condition-JLX}
\begin{split}
&\text{ for any } x\neq \0, 0<\lambda<|x|, |z|>\lambda \text{ and } a\leq b, \text{ there holds } \\
& \qquad f(|z|,a)> \big( \frac{\lambda}{|z|}\big)^{n+2} f\Big( \Big|x+\frac{\lambda^2 z}{|z|^2}\Big|, \big( \frac{|z|}{\lambda} \big)^{n-2} b\Big),
\end{split}
\end{equation}
the authors proved that any positive solution $u\in C^2(\R^n\setminus\{\0\})$ to \eqref{EqDelta} must be radially symmetric and monotone decreasing with respect to the origin.

In this paper, let us show the geometric extension to Eq. \eqref{EqDelta} with Laplace operator of arbitrary order $s\in \R_+ \setminus \{n/2\}$. Following \cite[Def. 2.2]{Sil07} (see also \cite{CFY15, CAL15}), we introduce the nonlocal operator, that is the fractional Laplacian $(-\Delta)^s$ in $\R^n$ with $s\in (-n/2,1]$, via the Fourier transform,  
\begin{equation}
\widehat{(-\Delta)^s f}(\xi)= |\xi|^{2s} \widehat f(\xi),
\end{equation}
for any function $f$ in the Schwartz space.  In the particular case $s \in (0,1)$, the operator $(-\Delta)^s$ can be presented via the singular integral, i.e. 
\[
\begin{split}
(-\Delta)^s f(x) &:= C_{n,s} P.V. \int_{\R^n} \frac{f(x)-f(y)}{|x-y|^{n+2s}}dy \\
&= C_{n,s}\lim_{\varepsilon \to 0^{+}} \int_{|y-x|\geq \varepsilon} \frac{f(x)-f(y)}{|x-y|^{n+2s}}dy,
\end{split}
\]
for some constant $C_{n,s}>0$.  Hence,  for non-integer $s>1$, the higher power fractional Laplacian can be understood as
\[
(-\Delta)^s = (-\Delta)^{s-[s]} (-\Delta)^{[s]},
\] 
where $[\cdot]$ is the usual floor function. We also refer to the extension method \cite{CS07} to define the fractional Laplacian.    With the fractional Laplacian $(-\Delta)^s$  defined above in $\R^n$, we consider the following general form of  elliptic equation \eqref{EqRn}, inspired by \cite{JLX08}, 
\begin{equation}\label{EqDeltaM}
(-\Delta)^s u=f(|x|,u) \quad\text{and}\quad u>0 \text{ in }\R^n\setminus\{\0\},
\end{equation}
for some smooth function $f$. The equation \eqref{EqDeltaM} is related the problem of prescribing $Q$-curvature in conformal geometry, that we present below to make our paper more comprehensive. 

First, let $s\in \N\setminus\{n/2\}$ be an integer, we consider the model $(\S^n, g_{\S^n})$ equipped with the standard metric $g_{\S^n}$ and let $\Delta_{\S^n}$ be the Laplace--Beltrami operator  on $\S^n$.   It was discovered by Graham, Jenne, Mason and Sparling \cite{GJMS92} that there exists a generalized operator of order $2s$ to  the well-known conformal Laplace operator 
\[
-\Delta_{\S^n} + \frac{n(n-2)}4
\]
in this setting. This operator is currently well-known as the GJMS operator. With the standard metric $g_{\S^n}$ on $\S^n$, we have the precise formula 
\begin{equation}\label{eq-GJMS}
P_{2s,g_{\S^n}} = \prod\limits_{k =1}^s { \Big( -\Delta_{\S^n} + \Big( \frac n2 - k \Big) \Big( \frac n2 + k - 1 \Big) \Big)}.
\end{equation}
The GJMS operator is conformal in the sense that with any metric $\widetilde g= v^{4/(n-2s)}g_{\S^n}$  for some smooth function $v$ on $\S^n$, we have the following relation between the two operators $P_{2s,\widetilde g}$ and $P_{2s,g_{\S^n}}$,
\begin{equation}\label{EqRelate}
P_{2s, \widetilde g}(\varphi)= v^{-\frac{n+2s}{n-2s}}P_{2s,g_{\S^n}}(v\varphi) 
\end{equation}
for any smooth, positive function $\varphi$ on $\S^n$.  We set $\varphi\equiv 1$ in  \eqref{EqRelate} to obtain 
\[
P_{2s,g_{\S^n}}(v) =P_{2s, \widetilde g}(1)v^{\frac{n+2s}{n-2s}}.
\]
Owing to \cite[Eq. (1.12)]{Juh13}, we have
\[
P_{2s, \widetilde g}(1)=(-1)^s \big(\frac{n}2-s\big)Q_{2s, \widetilde g}
\]
for some scalar function $Q_{2s,\widetilde g}$ known that the $Q$-curvature associated with the GJMS operator $P_{2s, \widetilde g}$. We thus obtain from \eqref{EqRelate} that
\begin{equation}\label{EqSn}
P_{2s, g_{\S^n}}(v)= K v^{\frac{n+2s}{n-2s}} \quad\text{on }\S^n,
\end{equation}
for some smooth function $K$ on $\S^n$.

Second, let $s\in \R_+ \setminus\{n/2\}$,   the fractional extension to the GJMS operator \eqref{eq-GJMS} fulfilling \eqref{EqRelate}  on $\S^n$  has been widely studied.  Repeating our previous work with Ng\^o \cite{LeNgoNguyen}, we briefly show the construction of  $P_{2s,g_{\S^n}}$ (see \cite[section 6.4]{Gon18} and \cite[section 3]{KL22}). Let $\N_0 := \N \cup \{ 0\}$ and $\{Y_l\}_{l \in \N_0}$ be the $L^2(\S^n)$-orthonormal basis of spherical harmonics of degree $l$.  In fact, all $Y_l$ are the eigenfunctions of the (negative) Laplace--Beltrami operator $-\Delta_{\S^n}$ on $\S^n$ corresponding to the eigenvalues $\lambda_l = l (l+ n-1)$ with $l \in \N_0$, namely 
\begin{equation}\label{eq-Yl}
-\Delta_{\S^n} Y_{l} = \lambda_l Y_{l}, \quad \lambda_l = l (l+ n-1).
\end{equation}
For any $v\in L^2(\S^n)$,  we have the unique representation
\[
v = \sum_{l \in \N_0} v_{l}Y_{l}.
\]
Let 
\begin{equation}\label{Alpha}
\alpha_{2s,n}(l)= \frac{\Gamma(l+n/2+s)}{\Gamma(l+n/2-s)}.
\end{equation}
the fractional operator $P_{2s,g_{\S^n}}$ is now defined by
\begin{equation}\label{eq-GJMS(u)}
P_{2s,g_{\S^n}} (v) = \sum_{l \in \N_0} \alpha_{2s,n}(l) v_{l}Y_{l},
\end{equation}
provided the right hand sides converges in $L^2 (\S^n)$. The fractional operator $P_{2s,g_{\S^n}}$ \eqref{eq-GJMS-frac} is well-defined in the whole range $s > 0$, provided that $\alpha_{2s,n}(l) = 0$ when the denominator in \eqref{Alpha} vanishes.   Equivalently,   the operator $P_{2s,g_{\S^n}}$ acts on $L^2(\S^n)$ by multiplication with $\alpha_{2s,n}(l)$.  An  alternative expression for $P_{2s,g_{\S^n}}$ is as follows
\begin{equation}\label{eq-GJMS-frac}
P_{2s,g_{\S^n}} =\frac{\Gamma(B+1/2+s)}{\Gamma(B+1/2-s)} \quad\text{with} \quad B=\sqrt{-\Delta_{\S^n}+\frac{(n-1)^2}4},
\end{equation}
where $\Gamma$ is the usual Gamma function. For $s$ integer, we obtain that 
\[
\begin{split}
P_{2s,g_{\S^n}} &= \big(B+s-\frac12 \big)\big(B+s-\frac32 \big)\dots \big(B-s+\frac32 \big)\big(B-s+\frac12 \big) \\
&= \prod_{k=1}^s \Big( B^2 - \big( \frac{2s-2k+1}2 \big)^2 \Big),
\end{split}
\]
that is exactly \eqref{eq-GJMS}. On spherical harmonics $Y_l$ of degree $l\in \N_0$, the operator $B$ acts by multiplication with $l+ (n-1)/2$ and therefore the operator $P_{2s,g_{\S^n}}$ acts by multiplication with $\alpha_{2s,n}(l)$ as stated \eqref{Alpha}.

With the above fractional GJMS operator $P_{2s,g_{\S^n}}$, we still consider the equation \eqref{EqSn}.    Let us write $N = (1,0,..,0) \in \S^n$ as the north pole, we recall  the stereographic projection $\pi_N :\S^n \to \R^n \setminus\{\0\}$ and  it is well-known that
\[
(\pi_N^{-1})^* (g_{\S^n}) = \big( \frac 2{1+|x|^2} \big)^2 dx^2.
\]
Making use of \cite[Lemma 4]{FKT} gives us that
\[
(-\Delta)^{\lfloor s \rfloor- s } \Big( \big( \frac 2{1+|x|^2} \big)^{\frac {n+2s}2} P_{2s,g_{\S^n}} (v) \circ \pi_N^{-1} \Big) = (-\Delta)^{\lfloor s \rfloor} \Big( \big( \frac 2{1+|x|^2} \big)^{\frac {n-2s}2} v \circ \pi_N^{-1} \Big)
\]
in $\R^n$. Therefore, if we let 
\begin{equation}\label{eq-u=v}
u(x) = \big( \frac 2{1+|x|^2} \big)^{\frac {n-2s}2} (v \circ \pi_N^{-1})(x),
\end{equation}
then $u$ solves
\begin{equation}\label{EqRn1}	
(-\Delta)^{\lfloor s \rfloor} u(x)= (-\Delta)^{\lfloor s \rfloor -s} ( h(|x|) u^{\frac{n+2s}{n-2s}}(x) ) \quad\text{in } \R^n\setminus\{\0\},
\end{equation}	
where $h$ is a smooth function on $\R_+$. Heuristically speaking, equation \eqref{EqRn1} can be seen as 
\begin{equation}\label{EqRn}
(-\Delta )^s u(x) =  h(|x|) u^{\frac{n+2s}{n-2s}}(x) \quad\text{in } \R^n\setminus\{\0\},
\end{equation}
which is of type \eqref{EqDeltaM}.
We refer to the arguments in \cite[page 8]{FKT} for the reason that we cannot rigorously claim this.  

Not only in conformal geometry, but also problems involving the fractional Laplacian operator $(-\Delta)^s$ of type \eqref{EqDeltaM} have been widely studied in other mathematical models. We refer to  \cite{BG90, Con06, CV10} and the references therein for the picture in numerous physical phenomena, such as anomalous diffusion and quasi-geostrophic flows, turbulence and water waves, molecular dynamics, and relativistic quantum mechanics of stars. Let us mention also the appearance of the (higher-order)fractional Laplacian  in  probability and finance \cite{CT04, App09}, and in  nonlinear elasticity and crystal dislocation \cite{Ant05, DFV14}.

Studying the radial symmetry of solutions to Eq. \eqref{EqDeltaM} is of our interest. However,  it is not easy to extend the condition \eqref{Condition-JLX} of Jin, Li and Xu \cite{JLX08} to the higher-order elliptic equation \eqref{EqDeltaM} even  for $s$ integer. The main reason is that the proof in \cite[Prop. 2.1]{JLX08} relies heavily on the maximum principle. So that, in this paper, we move to the corresponding integral equation of  Eq. \eqref{EqDeltaM},
\begin{equation}\label{eq-integral-neg}
u(x)= \int_{\R^n} |x-y|^p f(|y|,u(y)) dy \quad\text{in }\R^n\setminus\{\0\}.
\end{equation}
Indeed,  having  the fundamental solution of the fractional polyharmonic equation 
\[
(-\Delta)^{s} u = 0 \quad\text{in }\R^n \setminus \{ \0\}
\]
(see \cite[page 4]{LeNgoNguyen} and see also \cite[chapter 2]{CLM20}), we formally deduce the integral equation \eqref{eq-integral-neg} with $p=2s-n>0$. Assuming that $u\in C^1(\R^n\setminus\{\0\})$ is a positive solution to \eqref{eq-integral-neg},  we  now state the main result of this note.
\begin{theorem}\label{thm-great}
Let $n\geq 2$ and $p>0$. Let $f(\alpha,\beta): \R_+\times \R_+\to \R_+$ be a  continuous function satisfying the following condition  
\begin{equation}\label{Condition-F1}
\begin{split}
&\text{ for any }  x\neq \0, 0<\lambda<|x|, |z-x|>\lambda \text{ and } a\leq b, \text{ there holds } \\
& \qquad f(|z|,a)>\big(  \frac{\lambda}{|z-x|}\big)^{p+2n} f\Big( \Big| x+\frac{\lambda^2(z-x)}{|z-x|^2}\Big|, \big( \frac{\lambda}{|z-x|}\big)^{p} b\Big).
\end{split}
\end{equation}
Then, any positive solution $u \in C^1(\R^n\setminus\{\0\})$ to  \eqref{eq-integral-neg} must be radially symmetric and monotone increasing with respect to the origin.
\end{theorem}

Our main theorem reveals a sufficient condition of $f$ to show the radial symmetry of  positive solutions to  the general integral equation  \eqref{eq-integral-neg}  via the method of moving spheres.  Proving the radially symmetry property is an initial step in studying the integral equation \eqref{eq-integral-neg}, as well as the corresponding partial differential equation \eqref{EqDeltaM}. It has a key role in the classification of entire solutions (Liouville-type theorem) and in the characterization of their asymptotic behavior, see some celebrated papers \cite{CBS89, WeiXu, Li, CLO06} and  some recent papers \cite{Yan21, DPQ22, CDQ23, DQ23, HyderNgo}. 

We have the following remark from Theorem \ref{thm-great}, that is partly complement to the recent  result of Hyder-Ng\^o \cite{HyderNgo}. 
Indeed, they  consider the case 
\[
f(|y|,u)=f_\ve(|y|,u):= \ve \big(\frac1{1+|y|^2}\big)^{p+n} u+  \big( \frac1{1+|y|^2}\big)^{\frac{p+2n-pq}2} u^{-q} \quad (q>0).
\]
The authors make use of the method of moving planes to show that any smooth positive  solution $u$ to the integral equation 
\begin{equation}\label{EqHN}
u(x)=\int_{\R^n} |x-y|^p f_\ve(|y|,u(y))dy \quad\text{in }\R^n,
\end{equation}
appearing in a perturbative approach to study (sub)critical Sobolev inequalities on higher dimensional sphere,  is radially symmetric if one of the following conditions hold: $0<\ve \ll 1, 0<q \leq (p+2n)/p$ and $\ve=0, 0<q < (p+2n)/p$.  Clearly, any solution \eqref{EqHN} in $\R^n$ is also a solution to \eqref{EqHN} in $\R^n\setminus\{\0\}$.

We now  use the condition \eqref{Condition-F1} to revisit the result of Hyder-Ng\^o \cite{HyderNgo} as  $\ve=0$ and $q<(p+2n)/p$. That means, we let
\[
f(\alpha,\beta) =(1+\alpha^2)^{ \frac {pq-p-2n}2} \beta^{-q}.
\]
and use Theorem \ref{thm-great}  to deduce that for   any $x\neq \0, 0<\lambda<|x|, |z-x|>\lambda$ and  $a\leq b$,  there holds 
\[
(1+|z|^2)^{ \frac {pq-p-2n}2 } a^{-q} > \Big( 1+ \big|x+\frac{\lambda^2(z-x)}{|z-x|^2}\big|^2\Big)^{\frac {pq-p-2n}2 } \big(  \frac{\lambda}{|z-x|}\big)^{p+2n-pq} b^{-q}.
\]
This inequality is equivalent to 
\begin{equation}\label{IneHN}
(\frac{b}a)^q > \Big( \frac{\lambda^2}{|z-x|^2} \frac{1+|z|^2}{1+ \big|x+\frac{\lambda^2(z-x)}{|z-x|^2}\big|^2}   \Big)^{\frac{p+2n-pq}2}. 
\end{equation}
Thanks to \cite[Lemma 3.4]{LeNgoNguyen}, we have that for any $x\neq \0$, $|z-x|>\lambda$ and $\lambda \in (0,\sqrt{1+|x|^2})$,
\[
\frac{\lambda^2}{|z-x|^2} \frac{1+|z|^2}{1+ \big|x+\frac{\lambda^2(z-x)}{|z-x|^2}\big|^2} <1.
\]
Hence,  the inequality \eqref{IneHN} holds true if  $0<q<(p+2n)/p$. We thus have any $C^1$ positive solution to the integral equation \eqref{EqHN} in $\R^n\setminus\{\0\}$ is radially symmetry if  $0<q<(p+2n)/p$.

\section{Symmetry property of solutions to \eqref{eq-integral-neg}}
\label{sec-IE-classification}

\subsection{The method of moving spheres}
We introduce the preliminaries of the method of moving spheres to prove Theorem \ref{thm-great}.   In what follows, let $u>0$  be a $C^1$-solution to \eqref{eq-integral-neg} with $p>0$.

Let $B_r(x)$ be the ball with radius $r>0$ and center at $x \in \R^n$. Given the real parameter $\lambda>0$ and $x\in \R^n$, we denote by 
\[
\xi^{x,\lambda}= x+ \frac{\lambda^2(\xi-x)}{|\xi-x|^2} 
\] 
the inversion of $\xi \in \R^n\setminus\{x\}$ via the sphere $\partial B_\lambda(x) $.	
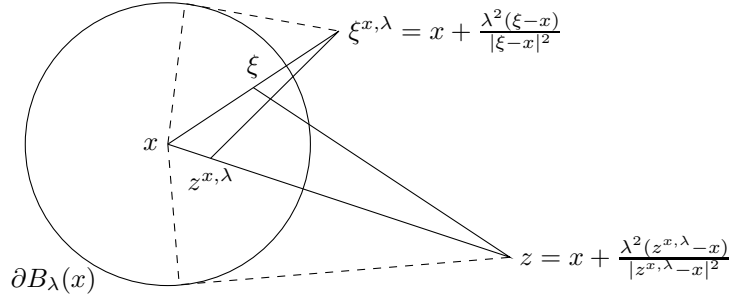
\begin{figure}[H]
\begin{tikzpicture}[scale=0.75]
\draw[-] (0,0) circle (2.5);
\draw[-] (0, 0) node[anchor = east] {$x$} 
-- (1.5,1) node[anchor =south] {$\xi$} 
-- (3,2) node[anchor =west] {$\xi^{x,\lambda} =x+ \frac{\lambda^2 (\xi-x)}{|\xi- x|^2}$} 
-- (0.75,-0.25);
\draw[dashed] (0,0) -- (0.3, 2.485) -- (3,2);
\draw[-] (0, 0) -- (0.75,-0.25) node[anchor =north] {$z^{x,\lambda}$} 
-- (6,-2) node[anchor =west] {$z = x+\frac{\lambda^2 (z^{x,\lambda} - x)}{|z^{x,\lambda}-x|^2}$}
-- (1.5,1) ;
\draw[dashed] (0,0) -- (0.2, -2.485) -- (6,-2);
\node [anchor =north] at (-2,-2) {$\partial B_\lambda (x)$};
\end{tikzpicture}
\caption[]{Inversion in the method of moving spheres.}\label{fig-MS}
\end{figure}
We also denote by $u_{x,\lambda}$ the Kelvin transform of $u$ via the sphere $\partial B_\lambda(x) $, namely 
\[
u_{x,\lambda} (\xi) = \big( \frac {|\xi-x|}{\lambda} \big)^{p} u (\xi ^ {x,\lambda}) = \big( \frac {|\xi-x|}{\lambda} \big)^{p} u \big( x+ \frac{\lambda^2(\xi-x)}{|\xi-x|^2} \big),
\]
for any $\xi \in \R^n\setminus\{x\}$.
Therefore
\begin{equation*}\label{eq-KelvinKelvin}
u_{x,\lambda} (\xi^ {x,\lambda}) = \big( \frac {|\xi^ {x,\lambda}-x|}{\lambda} \big)^{p} u (\xi)= \big( \frac {\lambda}{|\xi-x|} \big)^{p} u (\xi).
\end{equation*}
Moreover, it is clear that $(\xi^ {x,\lambda})^{x, \lambda}= \xi$ and 
\begin{equation}\label{eq-|Kelvin|}
|z-x| |\xi-x||\xi^{x,\lambda} - z^{x,\lambda}| = \lambda^2 |\xi-z|. 
\end{equation}
	
The basic idea of the method of moving spheres is to compare $u$ and $u_{x,\lambda}$ pointwise with small $\lambda >0$. Indeed, we will prove in Lemma \ref{lem-MS-small} the following inequality,
\begin{equation}\label{CompareU}
u_{x,\lambda} (y)\geq u(y)
\end{equation}
for any $y$ satisfying $|y-x| \geq \lambda> 0$ and  for sufficiently small $\lambda > 0$. 	 The proof of \eqref{CompareU} relies  on the  condition \eqref{Condition-F1} of $f$. After that, we increase $\lambda$ toward the largest value of possible to ensure that the inequality \eqref{CompareU} still holds true.

\subsection{A comparision lemma}

\begin{lemma}\label{lem-MS-small}
There exists some $\lambda_0 > 0$ such that for any $\lambda \in (0, \lambda_0)$, we have 
\[
u_{x,\lambda} (y)\geq u(y)
\] 
for any $y$ satisfying $|y-x| \geq \lambda> 0$.
\end{lemma}
\begin{proof}
This lemma will be proved in two steps.

In the first step, we fix $x\in \R^n$ and $\theta$ in the standard unit sphere $\S^n$ and compute that 
\begin{equation*}\label{x1}
\begin{aligned}
\frac{d}{dr} \big( r^{-\frac{p}{2}} u(x+ r\theta) \big) & =-\frac{p}{2} r^{-\frac{p}{2}-1} u(x+ r\theta) + r^{-\frac{p}{2}} (\nabla u) (x + r\theta) \cdot \theta \\
& = r^{-\frac{p}{2}-1} u(x + r\theta) \Big(- \frac{p}{2} + r \frac{\nabla u }{u} \Big|_{x + r\theta } \cdot \theta\Big)\\
& \leq r^{-\frac{p}{2}-1} u(x + r\theta) \Big(- \frac{p}{2} +  r \|\nabla \log u \|_\infty \Big),
\end{aligned}
\end{equation*}
because $u$ is positive and of class $C^1$. That means,  the mapping 
\begin{equation}\label{decrease}
r \mapsto r^{-\frac{p}2} u(x+r\theta)
\end{equation}
is decreasing in the regime 
\[
0<r<  \min\Big(1, \frac{p}{2\|\nabla\log u\|_\infty}\Big).
\]
Now,  let $y$ satisfy $0<\lambda< |y-x| \leq \lambda_1$. That implies
\[
|y^{x,\lambda}-x| \leq \lambda\leq |y-x|.
\]
Hence, choosing  $\theta= (y-x)/|y-x| \in \S^n$ and $r=|y-x|$, we deduce from the decrease \eqref{decrease} that
\[
|y-x|^{-\frac{p}2} u \big(x + |y-x| \frac{y-x}{|y-x|} \big) \leq \big| y^{x, \lambda} -x \big|^{-\frac{p}2} u \big( x + | y^{x, \lambda} - x | \frac{y-x}{|y-x|} \big).
\]
This shows that 
\begin{equation}\label{ine-leq}
u(y) \leq \Big( \frac{\big| y^{x, \lambda} -x \big|}{| y-x|} \Big)^{-\frac{p}2} u(y^{x, \lambda}) = u_{x,\lambda} (y)
\end{equation}
for any $y$ satisfying $0<\lambda< |y-x| \leq \lambda_1$.

In the second step, we consider $|y-x| \geq \lambda_1 > 0$. We will prove that 
\[
u_{x, \lambda}(y) \geq u(y) \quad\text{for }|y-x| \geq \lambda_1 > 0.
\] 
To do this, we will show that 
\begin{equation}\label{z}
\int_{\R^n} (1+|z|^p) f (|z|,u(z)) dz <\infty.
\end{equation}
Indeed, fix $\bar y \in \R^n$ such that $ 1 \leq|\bar y| \leq 2 $ to get that
$$u(\bar y)= \int_{\R^n} |\bar y-z|^{p} f (|z|,u(z))dz < \infty.$$
 It is easy to see that for $ z \in \R^n \setminus B_4(\0)$, we have
\[
|\bar y-z|\geq |z|-|\bar y|\geq \frac{|z|}2\geq 2.
\]
Hence, we get
\begin{equation*}\label{b}
\begin{aligned}
\int_{ \R ^n \setminus B_4(\0)}(1+|z|^p) f(|z|, u(z)) dz &\leq(1+2^p) \int_{ \R ^n \setminus B_4(\0)} |\bar y-z|^{p}f(|z|, u(z)) dz\\
& \leq (1+2^p) u(\bar y) < \infty.
    \end{aligned}
\end{equation*}
It follows from the continuity of $u$ and $f$ that 
\[
\int_{ B_4(\0)}(1+|z|^p) f(|z|, u(z)) dz \leq (1+4^p) \int_{ B_4(\0)}f(|z|, u(z)) dz<\infty.
\]
Thanks to the above inequalities, we obtain \eqref{z}. Thanks to \eqref{z}, we obtain for $|y| \geq 1$ that
\begin{equation}\label{eq-}
\int_{\R^n} \frac{|y-z|^{p}}{|y|^{p}} f(|z|, u(z))dz \leq \int_{\R^n} (1+|z|^{p}) f(|z|, u(z))dz <\infty.
\end{equation}
Thus, using  \eqref{eq-} and using the Lebesgue dominated convergence theorem, we deduce the following result
\[\begin{split}
\lim_{|y| \to \infty} |y|^{-p} u(y) &=  \lim_{|y| \to \infty} \int_{ \R ^n}\frac{|y-z|^{p}}{|y|^{p}} f(|z|, u(z)) dz \\
&= \int_{ \R ^n}f(|z|, u(z))dz \in (0, \infty),
\end{split}
\]
yielding
\[
u(y)\leq C |y-x|^{p}\; \;  \text{for all} \; \; \; |y-x| \geq \lambda_1 \; \text{and for some} \; C>0.
\]
Hence, for small $\lambda_0 \ll \lambda_1$ and for $\lambda \leq \lambda_0$ we have
\begin{equation}\label{ine-geq}
u_{x,\lambda} (y) = \big(\frac {|y-x|}{\lambda } \big)^{p} u (y^{x,\lambda})
\geq \Big(\frac {|y-x|}{\lambda_0 } \Big)^{p} \inf_{B_{\lambda_1} (x)} u \geq u(y)
\end{equation}
for $|y-x| \geq \lambda_1 > 0$. 

The two inequalities \eqref{ine-leq} and \eqref{ine-geq} help us to  complete the proof.
\end{proof}
Let $x \in \R^n \setminus \{\0\}$ be arbitrary, but fixed, we define
\begin{equation}\label{eq-lambda}
\ov \lambda (x):= \sup\left\{ 
\begin{aligned}
\mu>0, \quad & u_{x, \lambda}(y)\geq u(y) \quad \text{for all } \; |y-x| \geq \lambda >0\\
	& \text{and for all } \; 0< \lambda < \mu 
\end{aligned}
\right\}.
\end{equation}
Thanks to Lemma \ref{lem-MS-small}, we have that $\ov \lambda(x)$ is well defined and $0<\ov \lambda(x) \leq +\infty$.

\begin{proposition}\label{LeM=002}
There holds 
\[
\ov\lambda(x) \geq |x| \quad\text{for any }x\in \R^n\setminus\{\0\}. 
\]
\end{proposition}
To prove Proposition \ref{LeM=002}, we establish  integral formulations of $u_{x,\lambda}$ and $u_{x,\lambda}-u$ below.
\begin{lemma}\label{lem-ul}
There holds
\begin{align}\label{eq-u_{x,lambda}}
u_{x,\lambda}(\xi)=  \int_{ \R ^n}\frac{\lambda^{p+2n} |\xi-z|^{p}}{|z-x|^{p+2n}} f\big(|z^{x,\lambda}|,u(z^{x,\lambda}) \big)  dz,
\end{align}
for any  $\xi \in \R^n\setminus\{x\}$.
\end{lemma}
\begin{proof}
Using \eqref{eq-integral-neg} with $y=z^{x,\lambda}$, we get
\begin{align*}
u_{x,\lambda} (\xi) & = \big( \frac{|\xi-x|}{\lambda} \big)^{p} u(\xi^{x,\lambda})\\
&=  \big( \frac{|\xi-x|}{\lambda} \big)^{p} \int_{ \R ^n}  f(|z^{x,\lambda}|, u(z^{x,\lambda})) |\xi^{x,\lambda} -z^{x,\lambda}|^{p}d(z^{x, \lambda}).	
\end{align*}
Computing directly, one has
\begin{equation}\label{eq-dz=dz}
d(z^{x,\lambda}) = \big(\frac{\lambda }{|z-x|} \big)^{2n} dz.
\end{equation}
Therefore, we obtain
\[\begin{split}
u_{x,\lambda} (\xi)	&= \int_{ \R ^n} \big( \frac{|\xi-x|}{\lambda} \big)^{p}  |\xi^{x,\lambda} -z^{x,\lambda}|^{p} f\big(|z^{x,\lambda}|,u(z^{x,\lambda}) \big)  \big(\frac{\lambda }{|z-x|} \big)^{2n} dz.
\end{split}\]
Using the relation \eqref{eq-|Kelvin|},
we have
		\[\begin{split}
			u_{x,\lambda} (\xi)	&=  \int_{ \R ^n} \big( \frac{\lambda |\xi-z|}{|z-x|}\big)^{p}  \big(\frac{\lambda }{|z-x|} \big)^{2n}  f\big(|z^{x,\lambda}|,u(z^{x,\lambda}) \big)  dz\\
			&=  \int_{ \R ^n}\frac{\lambda^{p+2n} |\xi-z|^{p}}{|z-x|^{p+2n}}f\big(|z^{x,\lambda}|,u(z^{x,\lambda}) \big) dz.
		\end{split}
		\]
		Hence, we obtain our desired equality \eqref{eq-u_{x,lambda}}. The proof is completed.
	\end{proof}
	
	Having the formula \eqref{eq-u_{x,lambda}}, we obtain the integral form of  the difference between $u_{x,\lambda}$ and $u$.	
\begin{lemma}\label{lem-U-U}
We have
\begin{equation}\label{eq-u-u=FG}
\begin{split}
&u_{x,\lambda} (\xi)-u(\xi) \\
&=  \int_{|z-x| \geq \lambda}K(x, \lambda; \xi,z) \Big(f(|z|,u(z))-( \frac {\lambda}{|z-x|} )^{p+2n}f\big(|z^{x,\lambda}|,u(z^{x,\lambda}) \big)  \Big)dz.
\end{split}
\end{equation}
with the kernel $K(x, \lambda; \xi,z)$ given by
\begin{equation}\label{eq-K}
\begin{aligned}
K(x, \lambda; \xi,z)&=  \big(\frac {|\xi-x|}{\lambda} \big)^{p} |\xi^{x, \lambda}-z|^{p}- |\xi-z|^{p}\\
& =  |\xi-z^{x,\lambda}|^{p} \big( \frac{|z-x|}{\lambda} \big)^{p}-|\xi-z|^{p}.
\end{aligned}
\end{equation}		
Furthermore, the kernel $K> 0$ for any $|\xi-x|>\lambda$ and $|z-x| > \lambda$. 
\end{lemma}
\begin{proof}
Using \eqref{eq-integral-neg} again after decomposing the integral $\int_{ \R ^n}$  into two parts $ \int_{\{z:|z-x| \geq \lambda\}}$ and $ \int_{\{z:|z-x| \leq \lambda\}}$ to get
\begin{align*}
u (\xi)=&  \int_{\{z:|z-x| \geq \lambda\}} |\xi-z|^{p} f (|z|,u(z)) dz \\
&\qquad + \int_{\{z^{x, \lambda}:|z^{x, \lambda}-x| \leq \lambda\}} |\xi-z^{x, \lambda}|^{p} f (|z^{x, \lambda}|, u(z^{x,\lambda})) d(z^{x, \lambda}).
\end{align*}
For the integral $\int_{\{z^{x, \lambda}:|z^{x, \lambda}-x| \leq \lambda\}}$, we  make change of  variable $z^{x, \lambda} \mapsto z$  and use  \eqref{eq-dz=dz} to obtain	
\begin{align*}
u (\xi)=&  \int_{|z-x| \geq \lambda} |\xi-z|^{p}  f (|z|,u(z))  dz \\
&+  \int_{|z-x| \geq \lambda} |\xi-z^{x, \lambda}|^{p}  f (|z^{x, \lambda}|, u(z^{x, \lambda})) \big(\frac{\lambda }{|z-x|} \big)^{2n} dz.
\end{align*}
Then, we get
\begin{equation}\label{Decompose-U}
\begin{split}
u(\xi) =& \int_{|z-x| \geq \lambda} |\xi-z|^{p}  f (|z|,u(z))   dz \\
&+ \int_{|z-x| \geq \lambda } \frac{\lambda^{2n} |\xi-z^{x,\lambda}|^{p}}{|z-x|^{2n}}f\big(|z^{x,\lambda}|,u(z^{x,\lambda}) \big)   dz.
\end{split}
\end{equation}
Note that $(z^ {x,\lambda})^{x, \lambda}= z$. Using  \eqref{eq-u_{x,lambda}} in Lemma \ref{lem-ul}, we clearly have
\begin{align*}		
u_{x,\lambda}(\xi)= &  \int_{|z-x|\geq \lambda} \frac{\lambda^{p+2n} |\xi-z|^{p}}{|z-x|^{p+2n}} f\big(|z^{x,\lambda}|,u(z^{x,\lambda}) \big)   dz\\
&\qquad +  \int_{\{z^{x, \lambda} : |z^{x, \lambda}-x|\leq \lambda\}} \frac{\lambda^{p+2n} |\xi-z^{x,\lambda}|^{p}}{|z^{x,\lambda}-x|^{p+2n}} f\big(|z|,u(z) \big) d(z^{x, \lambda})\\
= &  \int_{|z-x|\geq \lambda} \frac{\lambda^{p+2n} |\xi-z|^{p}}{|z-x|^{p+2n}}f\big(|z^{x,\lambda}|,u(z^{x,\lambda}) \big)    dz\\
&\qquad +  \int_{|z-x|\geq \lambda} \frac{\lambda^{p+2n} |\xi-z^{x,\lambda}|^{p}}{|z^{x,\lambda}-x|^{p+2n}} f\big(|z|,u(z)\big)  d(z^{x,\lambda}).
 \end{align*}
Thanks to \eqref{eq-|Kelvin|} again, we have
 $$|z^{x,\lambda}-x||\xi-x||\xi^{x, \lambda}-(z^{x, \lambda})^{x, \lambda}|=|z^{x,\lambda}-x||\xi-x||\xi^{x, \lambda}-z|=\lambda^2 |\xi-z^{x, \lambda}|$$
 then
  $$ |\xi-z^{x, \lambda}|=\frac{|z^{x,\lambda}-x||\xi-x||\xi^{x, \lambda}-z|}{\lambda^2}.$$
 Therefore, using \eqref{eq-dz=dz} again also, we have 
\begin{align*}\label{Decompose-Ux}		
u_{x,\lambda}(\xi)= &  \int_{|z-x|\geq \lambda} \frac{\lambda^{p+2n} |\xi-z|^{p}}{|z-x|^{p+2n}} \; f\big(|z^{x,\lambda}|,u(z^{x,\lambda}) \big)   dz\\
&\quad+  \int_{ |z-x|\geq \lambda} \frac{\lambda^{p+2n} |z^{x,\lambda}-x|^{p}|\xi-x|^{p}|\xi^{x, \lambda}-z|^{p}}{\lambda^{2p}|z^{x, \lambda}-x|^{p+2n}} f\big(|z|,u(z) \big) \Big(\frac{\lambda}{|z-x|}\Big)^{2n}dz \\
= & \int_{|z-x|\geq \lambda} \frac{\lambda^{p+2n} |\xi-z|^{p}}{|z-x|^{p+2n}}\; f\big(|z^{x,\lambda}|,u(z^{x,\lambda}) \big)    dz\\
&\qquad+  \int_{|z-x|\geq \lambda} \frac{|\xi-x|^{p}|\xi^{x, \lambda}-z|^{p}} {\lambda^{p-4n}|z-x|^{2n}|z^{x, \lambda}-x|^{2n}} \; f\big(|z|,u(z)\big)  dz.\end{align*}
Note that $|z^{x, \lambda}-x||z-x|=\lambda^2$, we obtain
\begin{equation}\label{Uxlambda}
\begin{split}
u_{x,\lambda}(\xi)=	&  \int_{|z-x|\geq \lambda} \frac{\lambda^{p+2n} |\xi-z|^{p}}{|z-x|^{p+2n}}\; f\big(|z^{x,\lambda}|,u(z^{x,\lambda}) \big)    dz\\
&\qquad +  \int_{|z-x|\geq \lambda} \frac{|\xi-x|^{p}|\xi^{x, \lambda}-z|^{p}} {\lambda^{p}} \; f\big(|z|,u(z)\big)  dz.
\end{split}
\end{equation}

In view of \eqref{Decompose-U} and \eqref{Uxlambda}, we obtain 
\begin{align*}
u_{x,\lambda} (\xi)-u(\xi)=  \int_{|z-x| \geq \lambda} \Big(k_2  ( \frac {\lambda}{|z-x|} )^{p+2n} f\big(|z^{x,\lambda}|,u(z^{x,\lambda}) \big)-k_1 f(|z|,u(z)) \Big)dz.
\end{align*}
with $k_1, k_2$ being
\[
k_1= k_1(x, \lambda; \xi,z)= |\xi-z|^{p}-\big( \frac{|\xi-x|}{\lambda} \big)^{p} |\xi^{x, \lambda}-z|^{p}.
\]
and		
\[ 
k_2= k_2(x, \lambda; \xi,z)=|\xi-z|^{p}- |\xi-z^{x,\lambda}|^{p} \big( \frac{|z-x|}{\lambda} \big)^{p}.
\]
It can be seen that 
\begin{align*}
\big( \frac {|\xi-x|}{\lambda} \big)^2 |\xi^{x, \lambda}-z |^2 - |\xi-z |^2 & =\frac{(|z-x|^2-\lambda^2) (|\xi-x|^2-\lambda^2)}{\lambda^2}
\end{align*}
and that
\begin{align*}
\big( \frac {|z-x|}{\lambda} \big)^2 |\xi-z^{x, \lambda} |^2 - |\xi-z |^2 = \frac{(|z-x|^2-\lambda^2) (|\xi-x|^2-\lambda^2)}{\lambda^2}.
\end{align*}
Hence, the negativity of the kernels $k_1<0$ and $k_2 < 0$ for any $|\xi-x|>\lambda$ and $|z-x|>\lambda$ follows. In addition, the above calculations also tell us that $k_1=k_2$. Hence, by putting $k_1=k_2=-K$, we obtain  \eqref{eq-u-u=FG}. The proof is complete.
	\end{proof}

We are in position to prove Proposition \ref{LeM=002} now.	
\begin{proof}[Proof of Proposition \ref{LeM=002}]
By way of contradiction, we suppose  $\ov\lambda(x_0)<|x_0|$ for some points $x_0\in \R^n\setminus\{\0\}$. To simplify, we can  write 
\[
\ov\lambda_0=\ov\lambda(x_0) \quad\text{and}\quad \ov\delta_0 = \min\{1, \frac{|x_0|-\ov \lambda_0}2\}>0. 
\]
From the definition of $\ov\lambda_0$, we know that 
\begin{equation}\label{5CL6-01}
u_{x_0, \ov\lambda_0} (y) \geq u(y) \quad \text{for all $y$ satisfying } |y - x_0|\geq \ov\lambda_0.
\end{equation} 
However, for some $\lambda$ slightly bigger than $\ov\lambda_0$, we will prove that 
\[
u_{x_0,\lambda}(y) \geq u(y) \quad \text{for all $y$ satisfying } |y - x_0|\geq \lambda.
\]
In this situation, we obtain a contradiction to the definition of $\ov\lambda_0$ in \eqref{eq-lambda}.  We thus get $\ov\lambda(x_0)\geq |x_0|$. The proof follows from the estimates $u_{x_0,\lambda}-u$ outside and inside the ball $B(x_0,\ov\lambda_0+\ov\delta_0)$.
		
\noindent\textbf{Estimate of $u_{x_0,\lambda}-u$ outside the ball $B(x_0,\ov\lambda_0+\ov\delta_0)$}. To do that, we first estimate $u_{x_0,\ov\lambda_0}-u$. In the region $|y-x_0|\geq \ov\lambda_0+\ov\delta_0$, we  prove that there exists $C_0\in (0,1)$ such that 
\begin{equation}\label{Lhhdy6-03=001}
(u_{x_0,\ov\lambda_0}-u)(y) \geq C_0|y-x_0|^{p} \quad\text{for all } |y-x_0|\geq \ov\lambda_0+\ov\delta_0.
\end{equation}
Let us use the identity \eqref{eq-u-u=FG} in Lemma \ref{lem-U-U} to obtain
\begin{equation}\label{eq-u-u>=}
\begin{split}
&( u_{x_0,\ov\lambda_0}-u) (y) \\
&=  \int_{|z-x_0| \geq \ov\lambda_0} K(x_0, \ov\lambda_0; y, z) 
\Big(    f(|z|,u(z))-( \frac {\ov\lambda_0}{|z-x_0|} )^{p+2n}f\big(|z^{x_0,\ov\lambda_0}|,u(z^{x_0,\ov\lambda_0}) \big)  \Big)dz.
\end{split}
\end{equation}
Let
\begin{equation}\label{HVeLambda1}
H_{ x_0, \ov\lambda_0}(z) =   f(|z|,u(z))- (\frac {\ov\lambda_0}{|z-x_0|} )^{p+2n} f\big(|z^{x_0,\ov\lambda_0}|,u(z^{x_0,\ov\lambda_0}) \big),
\end{equation}
which is positive thanks to the condition \eqref{Condition-F1}.
Hence, using the Fatou's lemma, we have
\begin{align*}
\liminf_{|y| \nearrow \infty}  \frac{(u_{x_0,\ov\lambda_0}-u) (y)}{|y-x_0|^{p}} &= \liminf_{|y| \nearrow \infty} \int_{|z-x_0| \geq \ov\lambda_0} \frac{K(x_0, \ov\lambda_0; y, z)}{|y-x_0|^{p}} H_{x_0,\ov\lambda_0}(z) dz\\
			&\geq  \int_{|z-x_0| \geq \ov\lambda_0} \Big(  \Big(\frac{|z-x_0|}{\ov\lambda_0} \Big)^{p} -1 \Big) H_{x_0,\ov\lambda_0}(z) dz \\
			&> 0.	
		\end{align*}
		Thus, there is some small $C_1 > 0$ and some large $R \geq \ov\delta_0$ such that
		\[
		( u_{x_0, \ov\lambda_0}-u) (y) \geq C_1|y -x_0|^{p} \quad\text{for all } \; |y -x_0|\geq \ov\lambda_0 + R. 
		\]
		To obtain \eqref{Lhhdy6-03=001}, we are left to show that $ u_{x_0, \ov\lambda_0}-u$ is bounded from below by a positive constant in the ring $\{ y : \ov\lambda_0 +\ov\delta_0 \leq |y-x_0| \leq \ov\lambda_0+ R\}$. Thanks to the positivity of the kernel $K(x_0,\ov\lambda_0; y,z)$ in the region $|y-x_0| > \ov\lambda_0$ and $|z-x_0| > \ov\lambda_0$ (see again Lemma \ref{lem-U-U}), there is some small $C_2>0$ such that
		\[
		K(x_0, \ov\lambda_0; y ,z) \geq C_2
		\]
		for all $y$ and $z$ satisfying
		\[ 
		\ov\lambda_0 +\ov\delta_0 \leq |y-x_0| \leq \ov\lambda_0 + R < 2(\ov\lambda_0 + R) \leq |z-x_0| \leq 4(\ov\lambda_0 + R). 
\]	
		Using this and \eqref{eq-u-u>=1} we can estimate 
		\begin{align*}
			(u_{x_0,\ov\lambda_0}-u) (y) 
			\geq  C_2 \int_{2(\ov\lambda_0 + R) \leq |z-x| \leq 4(\ov\lambda_0 + R)} H_{x_0,\ov \lambda_0}(z) dz
			= C_3
		\end{align*}
		for some $C_3 >0$ which  depends only on $x_0$. Combining the above estimates, we deduce \eqref{Lhhdy6-03=001} for some $C_0$ depending on $C_1$ and $C_3$.
		
		We continue to estimate $ u_{x_0, \lambda}-u$ in the region $|y-x_0| \geq \ov\lambda_0 +\ov\delta_0$ with $\ov\lambda_0 \leq \lambda \leq \ov\lambda_0 +\delta_1$ for some small $\delta_1 \in (0,\ov\delta_0)$. Indeed, since
		\[
		u_{x_0,\lambda} (y) = \big( \frac{|y-x_0|}{\lambda} \big)^{p} u(y^{x_0,\lambda}),
		\]
		we use the continuity to get 
		\begin{equation}\label{Ine3.12}
			( u_{x_0, \ov\lambda_0} - u_{x_0, \lambda}) (y) \geq -\frac{C_0}{2} |y -x_0|^{p}
		\end{equation}
		for some small $\delta_1 \in (0, \ov\delta_0)$ and for all $\ov\lambda_0 \leq \lambda \leq \ov\lambda_0 +\delta_1$ and all $|y-x_0| \geq \ov \lambda_0 + \ov \delta_0$. Here the constant $C_0$ is as in \eqref{Lhhdy6-03=001}. The two inequalities \eqref{Lhhdy6-03=001} and \eqref{Ine3.12} implies that 
\begin{equation}\label{eq-outside1}
( u_{x_0, \lambda}-u) (y) =( u_{x_0, \ov\lambda_0}-u) (y)+ ( u_{x_0, \lambda}-u_{x_0,\ov\lambda_0} ) (y) \geq \frac{C_0}2 |y -x_0|^{p}
\end{equation}
for all $y$ satisfying $|y -x_0|\geq \ov\lambda_0 + \ov\delta_0 $ and all $\lambda$ satisfying $\ov\lambda_0 \leq \lambda \leq \ov\lambda_0 +\delta_1$.
		
\noindent\textbf{Estimate of $ u_{x_0, \lambda}-u$ inside the ball $B(x_0, \ov\lambda_0 +\ov\delta_0)$}. With the constant $\delta_1$ being found in the previous step, we continue our analysis  assuming  $\ov\lambda_0 \leq \lambda \leq \ov\lambda_0 +\delta_1$ to  obtain \eqref{eq-inside1} below.
		
Let $H_{x_0,\lambda}$ be  as same as $H_{x_0,\ov\lambda_0}$ (see \eqref{HVeLambda1}) after replacing $\ov\lambda_0$ by $\lambda$. We have that
\begin{equation}\label{eq-u-u>=1}
\begin{split}
(u_{x_0,\lambda}-u) (y) & \geq  \int_{|z-x_0| \geq \lambda} K(x_0, \lambda; y, z) H_{x_0,\lambda }(z)dz\\
&= \Big( \int_{ \ov\lambda_0 +\ov\delta_0 \geq |z-x_0| \geq \lambda} + \int_{ |z-x_0| \geq \ov\lambda_0 +\ov\delta_0 } \Big) K(x_0, \lambda; y, z) H_{x_0,\lambda}(z) dz.
\end{split}
\end{equation}
Due to \eqref{eq-outside1}, we have that $ u_{x_0,\lambda}\geq u$ in the region $|z-x_0| \geq \ov\lambda_0 +\ov\delta_0 $ and $\ov \lambda_0 \leq \lambda \leq \ov \lambda_0 + \delta_1$, yielding
\[
\int_{ |z-x_0| \geq \ov\lambda_0 +\ov\delta_0} K(x_0, \lambda; y, z) H_{x_0, \lambda}(z) dz \geq 
\int_{\ov\lambda_0 + 3 \geq |z-x_0| \geq \ov\lambda_0 + 2} K(x_0, \lambda; y, z) H_{x_0, \lambda}(z) dz.
\]
As a result,  we obtain that
\begin{align*}
(u_{x_0,\lambda}-u) (y) &\geq  \Big( \int_{ \ov\lambda_0 +\ov\delta \geq |z-x_0| \geq \lambda} + \int_{\ov\lambda_0 + 3 \geq |z-x_0| \geq \ov\lambda + 2} \Big) K(x_0, \lambda; y, z) H_{x_0, \lambda}(z) dz \\
& =  ( I + II ).
\end{align*}
Hence, let $\delta_2 \in (0, \delta_1)$ be sufficiently small, we will prove that (see \eqref{eq-inside1})
\[
I + II \geq 0 \; \; \; \;  \text{for all } \lambda \leq |y-x_{0}| \leq \ov\lambda_0 + \ov\delta_0, \; \ov\lambda_0 \leq \lambda \leq \ov\lambda_0 + \delta_2.
\]
		
\noindent\textit{Estimate of $I$}. 
Thanks to the smoothness of $f$ and $u$, there exists  $C_5>0$ independent of $\delta_2 $ such that
\[ 
\begin{split}
&\Big| f\Big(|z^{x_0, \lambda}|,  \big( \frac{\lambda}{|z-x_0|}\big)^{p} u_{x_0,\lambda}(z) \Big) - f\Big(|z^{x_0, \lambda}|,  \big( \frac{\lambda}{|z-x_0|}\big)^{p}u_{x_0,\ov\lambda_0}(z) \Big)\Big| \\
&\leq C_4\big(\frac{\lambda}{|z-x_0|}\big)^{p}|u_{x_0, \lambda }(z)- u_{x_0,\ov\lambda_0}(z)|\\
&\leq C_5 (\lambda - \ov\lambda_0 ) \leq C_5\delta_2,
\end{split}
\]
for all $\lambda \leq |z-x_0| \leq \ov\lambda_0 + \ov\delta_0$ and all $\ov\lambda_0 \leq \lambda \leq \ov\lambda_0 +\delta_2 $. Note that $u(z)\leq u_{x_0,\ov\lambda_0}(z)$ (see \eqref{5CL6-01}). Hence, combining the resulting inequality with the condition \eqref{Condition-F1} gives us that
\begin{equation}\label{eq-estimate-U-U1}
\begin{split}
H_{x_0,\lambda}(z) &=f(|z|,u(z)) - ( \frac {\lambda}{|z-x_0|} )^{p+2n}f\Big(|z^{x_0, \lambda}|,  \big( \frac{\lambda}{|z-x_0|}\big)^{p} u_{x_0,\lambda}(z) \Big)\\
&>  ( \frac {\lambda}{|z-x_0|} )^{p+2n}f\Big(|z^{x_0, \lambda}|,  \big( \frac{\lambda}{|z-x_0|}\big)^{p} u(z) \Big)\\&\qquad\qquad- ( \frac {\lambda}{|z-x_0|} )^{p+2n}f\Big(|z^{x_0, \lambda}|,  \big( \frac{\lambda}{|z-x_0|}\big)^{p} u_{x_0,\lambda}(z) \Big) \\
&>  ( \frac {\lambda}{|z-x_0|} )^{p+2n}f\Big(|z^{x_0, \lambda}|,  \big( \frac{\lambda}{|z-x_0|}\big)^{p} u_{x_0,\ov\lambda_0}(z) \Big)\\
&\qquad\qquad- ( \frac {\lambda}{|z-x_0|} )^{p+2n}f\Big(|z^{x_0, \lambda}|,  \big( \frac{\lambda}{|z-x_0|}\big)^{p} u_{x_0,\lambda}(z) \Big) \\
&\geq -C_6 \delta_2.
\end{split}
\end{equation}
Meanwhile, we utilize \eqref{eq-K} to get 
\begin{align}\label{eq-estimate-K11}
\int_{\lambda \leq |z-x_0| \leq \ov\lambda_0 +\ov\delta_0 } & K(x_0,\lambda; y, z) dz 
\leq C_7 (|y-x_0| - \lambda) 
\end{align}
for some $C_7>0$ (see Appendix \ref{apd-estimate-kernel}). 
In view of the two estimates \eqref{eq-estimate-U-U1} and \eqref{eq-estimate-K11}, we deduce
\begin{equation}\label{Est-I1}
			I = \int_{ \ov\lambda_0 +\ov\delta_0 \geq |z-x_0| \geq \lambda} K(x_0, \lambda; y, z) H_{x_0, \lambda}(z) dz 
			\geq - C_6C_7 \delta_2 (|y-x_0| - \lambda) .
		\end{equation}

\noindent\textit{Estimate of $II$}. Second, we estimate $II$ as follows.  
For any $\ov\lambda_0 \leq \lambda \leq \ov\lambda_0 +\delta_2$ and for any $z$ satisfying $\ov\lambda_0 +2 \leq |z-x_0| \leq \ov\lambda_0 +3$, we obtain from the condition \eqref{Condition-F1} again that $H_{x_0,\lambda}(z) >0$. By the continuity of $f$ and $u$, that implies
\[
\begin{split}
\min_{\ov\lambda_0 +2 \leq |z-x_0| \leq \ov\lambda_0 +3}H_{x_0,\lambda}(z) = C_8>0.
\end{split}
\]
Hence, we obtain 
\begin{equation}\label{IIgeq1}
\begin{split}
II &= \int_{\ov\lambda_0 + 3 \geq |z-x_0| \geq \ov\lambda_0 + 2} K(x_0, \lambda; y, z) H_{x_0, \lambda}(z) dz \\
& \geq C_8 \int_{\ov\lambda_0 + 3 \geq |z-x_0| \geq \ov\lambda_0 + 2} K(x_0, \lambda; y, z) dz .
\end{split}
\end{equation}
Note that $K(0, \lambda; \xi,z) =0$ for $\xi \in \partial B_\lambda(\0)$ from and that
\[\begin{split}
\langle \nabla_\xi K(0, \lambda; \xi, z), \xi \rangle \big|_{|\xi|=\lambda}&=p |\xi-z|^{p-2} ( |z|^2- |\xi|^2 ) \\
&> p |\xi-z|^{p-2}( (\overline\lambda_0 + 2)^2 - (\overline\lambda_0 + 1)^2) >0,
\end{split}\]
for all $\ov\lambda_0 + 2 \leq |z| \leq \ov\lambda_0 +3$. Therefore, there is some $C_9>0$ independent of $\delta_2$ such that
\begin{equation}\label{Est-K0}
K(0, \ov\lambda_0; \xi ,z) \geq C_9 (|\xi| - \lambda)
\end{equation}
for all $\xi$ near $\partial B_\lambda(\0)$ and all $z$ satisfying $\overline\lambda_0 + 2 \leq |z| \leq \overline\lambda_0 +3$. However, thanks to the positivity and the smoothness of $K$, we can choose $C_9$ even smaller, if necessary, in such a way that the preceding estimate also holds for all $\xi \in B_{\overline \lambda_0 + \overline \delta_0}(x) \setminus B_\lambda(x)$. Now letting $\xi = y-x_0$, we arrive at
\begin{equation}\label{Est-Kx_01}
K(x_0, \lambda; y ,z) = K(0, \lambda; y -x_0,z) \geq C_9 (|y-x_0| - \lambda)
\end{equation}
for all $\overline\lambda_0 + 2 \leq |z-x_0| \leq \overline\lambda_0 +3$ and all $ \lambda \leq |y-x_0| \leq \overline\lambda_0 + \overline\delta$. Putting \eqref{Est-Kx_01} into \eqref{IIgeq1}, we arrive at
\begin{equation}\label{Est-II1}
II \geq C_8 C_9 \Big(\int_{\ov\lambda_0 + 3 \geq |z-x_0| \geq \ov\lambda_0 + 2} dz \Big) (|y-x_0| - \lambda)=:C_{10} (|y-x_0| - \lambda).
\end{equation}
We are now in position to combine  \eqref{Est-I1} and \eqref{Est-II1}. By \eqref{Est-I1}, if we choose $\delta_2$ sufficiently small, then we have that 
\begin{equation}\label{eq-inside1}
\begin{split}
( u_{x_0,\lambda}-u) (y) &=I+II \geq ( -C_6 C_7 \delta_2 + C_{10}) ( |y -x_0| -\lambda)> 0,
\end{split}
\end{equation}
for $\ov\lambda_0 \leq \lambda \leq \ov\lambda_0 + \delta_2$ and for $\lambda \leq |y -x_0| \leq \ov\lambda_0 + \ov\delta_0 $.

With the help of  \eqref{eq-outside1} and \eqref{eq-inside1}, we deduce that $( u_{x_0,\lambda}-u) (y) >0$. This violates the definition of $\ov\lambda_0 $ in \eqref{eq-lambda}. Hence, Proposition \ref{LeM=002} is proven.
\end{proof}
	
\subsection{Proof of Theorem \ref{thm-great}}
Let us complete the proof of Theorem \ref{thm-great}.
\begin{proof}
By Lemma \ref{LeM=002}, we have for every $x\in \R^n \setminus \{\0\}$, 
\begin{equation}\label{SS555-1}
u(y) \leq u_{x,\lambda}(y) \quad \text{for all } \; |y - x| \geq \lambda, \; 0< \lambda< |x|. 
\end{equation}
Let $y\in \R^n \setminus \{\0\}$	and $a >0$ be arbitrary but fixed. So, for $ e$ be any unit vector in $\R^n$ (e.g. $ e = -y/|y|$), which satisfies
\begin{equation}\label{eq-vector1}
\langle y - a  e,  e\rangle \leq 0.
\end{equation}
 For any $R >a$, we set $\lambda = R - a>0$ and $x=R  e\in \R^n$. For $0<\lambda< |x|$, then 
\[
|y-x|^2 = |y - a  e -\lambda e|^2 =\lambda^2 + |y- a  e |^2 -2\lambda\langle y - a  e ,  e\rangle \geq \lambda^2,
\]
due to \eqref{eq-vector1}. Thus, using \eqref{SS555-1} to get
\begin{align*}
u(y)\leq u_{x,\lambda}(y) & =\big( \frac{|y - x|}{\lambda}\big)^{p} u\big( x + \frac{\lambda^2(y -x)}{|y -x|^2} \big) \\
& = \big(\frac{|y-R  e|}{R-a}\big)^{p} u\big(R  e+\frac{(R-a)^2(y-R  e)}{|y-R  e|^2}\big).
\end{align*}
Note that	
\[
\begin{split}
R  e+\frac{(R-a)^2(y-R  e)}{|y-R  e|^2} 
&= \frac{R(|y|^2-2R\langle y,  e\rangle+R^2)  e+ (R^2-2Ra+a^2)(y-R  e)}{|y-R  e|^2}.
\end{split}
\]
Therefore
\[
R  e+\frac{(R-a)^2(y-R  e)}{|y-R  e|^2} \to y-2(\langle y,  e \rangle-a)  e \quad\text{as } R\to+\infty.
\]
Moreover,
\[	 
\big(\frac{|y-R  e|}{R-a}\big)^{p} \to 1 \quad\text{as } R\to+\infty.
\]
Following the continuity of $u$, we get
\begin{equation}\label{SS666-1}
u(y) \leq u(y - 2( \langle y,  e \rangle - a )  e ).
\end{equation}
In the inequality \eqref{SS666-1}, we take the limit $a \searrow 0$ to obtain
\[
u(y)\leq u(-y) \quad \text{for all } \; y \in \R^n \setminus \{\0\}.
\]
This inequality implies that $u$ is radially symmetric about the origin. 
		
To complete the proof, we will show that $u$ is monotone increasing with respect to the origin. Let $y=(y_1,y_2,...,y_n)$, due to the radial symmetry of $u(y)$, it suffices to show the monotonicity in the $y_1$-direction.  For arbitrary but fixed $a > 0$, let us choose $ e = (1, 0,..., 0) \in \R^n$ and consider $y = (y_1, 0,\dots,0) \in (0, +\infty) \times \R^{n-1}$ in such a way that it satisfies \eqref{eq-vector1},  namely $0 < y_1 < a$. Then, by \eqref{SS666-1} and by 
\[
y - 2( \langle y,  e \rangle - a ) e = (2a-y_1, 0,\dots,0).
\]
we obtain
\[
u(y_1,0, ..., 0) \leq u(2a-y_1, 0,...,0).
\]
Hence, the function $t \mapsto u(t, 0, ..., 0)$ is increasing on $(0, +\infty)$. This implies that the symmetric function $u$ is monotone increasing with respect to the origin. The proof of Theorem \ref{thm-great} is complete.
\end{proof}

\section*{Acknowledgements}		
	
Both authors are deeply grateful to Assoc. Prof. Qu\cfac oc Anh Ng\^o for his fruitful discussions, improving the presentation of this paper. This paper was complete when the authors were working at the Vietnam Institute for Advanced Study in Mathematics (VIASM). We wish to thank VIASM for the financial support and the kind hospitality.

\appendix

\section{Estimate (\ref{eq-estimate-K11}) for the kernel $K$}\label{apd-estimate-kernel}
	
This appendix is to prove \eqref{eq-estimate-K11} , that is
\begin{align*}
\int_{\lambda \leq |z-x| \leq \overline\lambda +\overline\delta } K(x,\lambda; y, z) dz \leq C \big( |y-x| - \lambda \big) 
\end{align*}
for some $C>0$. The proof is standard, and we show that to make our paper fully comprehensive. 
	
Let us recall the elementary inequality
\begin{equation}\label{element}
| a^s -b^s | \leq s |a-b| \max\Big\{ a^{s-1},  b^{s-1} \Big\}
\end{equation}
for any $a,b>0$ and any $s>0$. Next we prove that with $\kappa > n$, there holds
\begin{equation}\label{eq-eq}
\int_{\lambda \leq |z-x| \leq \overline\lambda +\overline\delta } |y -z|^\kappa dz \leq C(\overline\lambda, \overline\delta,|x|,R)
\end{equation}
if $y \in \overline B(x,R)$. Indeed, \eqref{eq-eq} follows from the following estimate
\[
\begin{split}
\int_{\lambda \leq |z-x| \leq \overline\lambda +\overline\delta } |y -z|^\kappa dz 
& =\int_{\{y+z : \lambda \leq |z-x| \leq \overline\lambda +\overline\delta \}}|z|^\kappa dz
\\
&\leq \int_{B( \overline\lambda +\overline\delta +|x| +R)} |z|^\kappa dz
=: C
\end{split}
\]
by triangle inequality. Thanks to \eqref{eq-K}, we obtain
\begin{align*}
&\int_{\lambda+\delta_2 \leq |z-x| \leq \overline\lambda +\overline\delta }  K(x,\lambda; y, z) dz\\
& \leq  \int_{\lambda \leq |z -x| \leq \overline\lambda +\overline\delta } \Big| |y -z|^p - |y^{x,\lambda} -z|^p\Big| dz \\ 
&\qquad\qquad+\int_{\lambda \leq |z-x| \leq \overline\lambda +\overline\delta } \Big| \Big(\frac{|y-x|}{\lambda}\Big)^p -1 \Big| |y^{x,\lambda} -z|^p dz\\
&= I_1 + I_2.
\end{align*}
We estimate $I_1$ and $I_2$ term by term. For the integral term $I_1$, using \eqref{element}, we have
\[
\begin{split}
&\Big| |y -z|^p - |y^{x,\lambda} -z|^p\Big| \leq p |y-y^{x,\lambda}|  \max\Big\{ |y -z|^{p-1}, |y^{x,\lambda} -z|^{p-1}\Big\}.
\end{split}
\]
Note that for $y$ satisfying $\lambda \leq |y-x| \leq \overline\lambda +\overline \delta$,
\[
|y^{x,\lambda} - x| = \frac{\lambda^2}{|y-x|} \leq \lambda \leq \overline\lambda +\overline \delta.
\]
 Hence,  making use of \eqref{eq-eq} twice, we get for some $C>0$ that
\begin{equation}\label{EstI1}
I_1 \leq C |y-y^{x,\lambda}|.
\end{equation}
For the integral term $I_2$, as $|y-x| \geq \lambda$, we have
\begin{align*}
\Big| \Big(\frac{|y-x|}{\lambda} \Big)^p -1 \Big|
& =\lambda^{-p} \Big| |y-x|^p - \lambda^p \Big|\\
& \leq p \lambda^{-p} \big| |y-x| - \lambda \big| \max \Big\{|y-x|^p , \lambda^p \Big\}\\
& \leq p \big( |y-x| - \lambda \big).
\end{align*}
 Thus, we continue applying \eqref{eq-eq} to obtain for some $C>0$
\begin{equation}\label{EstI2}
I_2 \leq C \big( |y-x| - \lambda \big).
\end{equation}
Combining \eqref{EstI1} and \eqref{EstI2} gives us that
\begin{align*}
\int_{\lambda \leq |z-x| \leq \overline\lambda +\overline\delta } K(x,\lambda; y, z) dz & \leq C |y-y^{x,\lambda}| + C \big( |y-x| - \lambda \big) \\
&=C \Big| 1- \frac{\lambda^2}{|y-x|^2} \Big| \big|y-x\big| + C \big( |y-x| - \lambda \big) \\
& \leq C \big( |y-x| - \lambda \big) 
\end{align*}
for some $C>0$, which is our desired inequality.

\end{document}